\documentclass[a4paper,11pt,oneside]{amsart}
\usepackage{amsbsy,amssymb,pstricks,pst-node,mathrsfs,MnSymbol}
\usepackage{mathtools,graphicx}
\usepackage{subfigure,hyperref}
\usepackage[shortlabels]{enumitem}
\usepackage[pagewise]{lineno}
\usepackage{tikz}
\usetikzlibrary{positioning}
\newtheorem{theorem}{Theorem}[section]
\newtheorem{thm}[theorem]{Theorem}
\newtheorem{lem}[theorem]{Lemma}

\newtheorem{cor}[theorem]{Corollary}
\newtheorem{example}[theorem]{Example}

\theoremstyle{definition}
\newtheorem{definition}[theorem]{Definition}

\newcommand{\GP}{\Gamma(P)}

\begin{document}

	\title{Cohen-Macauleyness of the Zero-Divisor Graph of a Boolean poset }

\maketitle

\markboth{P. Waghmare and V. Joshi}{Cohen-Macauleyness of the zero-divisor graph of a Boolean poset }\begin{center}\begin{large}Prashant Waghmare$^\text{a}$ and Vinayak Joshi$^\text{b}$\end{large}\\\begin{small}\vskip.1in$^\text{a}$\emph{Department of Mathematics, Maharashtra Udayagiri Mahavidyalaya,\\ Udgir - 41517, Maharashtra, India}\\$^\text{b}$\emph{Department of Mathematics, Savitribai Phule Pune University,\\ Pune - 411007, Maharashtra, India\\E-mail: prashantwagh49@gmail.com (P. Waghmare), vvjoshi@unipune.ac.in (V. Joshi)}\end{small}\end{center}\vskip.2in
	\begin{abstract}
		In this paper, we prove that the zero-divisor graph $\Gamma(P)$ of a Boolean poset $P$ is both well-covered and Cohen--Macaulay. Furthermore, for a poset $\mathbf{P} = \prod_{i=1}^{n} P_i$ $(n \ge 3)$, where each $P_i$ is a finite bounded poset satisfying $Z(P_i) = \{0\}$ for all $i$, and 
		$
		2 \le |P_1| \le |P_2| \le \cdots \le |P_n|,
	$
		we show that the zero-divisor graph $\Gamma(\mathbf{P})$ is Cohen--Macaulay if and only if $\mathbf{P}$ is a Boolean lattice.
		\end{abstract}

\noindent\begin{Small}\textbf{Mathematics Subject Classification (2020)}:  06E20, 13A70, 13C14    \end{Small}\vskip.2in
\vskip.25in

\section{Introduction}
In recent years, there has been a growing interest in the study of zero-divisor graphs of ordered structures. One of the main reason is that it serve as a tool to study various other graphs associated with algebraic structures, such as the zero-divisor graph of reduced commutative rings, the intersection graph of ideals, the comaximal ideal graph, the nilpotent graph, and the annihilating ideal graph.  

 The concept of a zero-divisor graph was introduced by I. Beck \cite{Be} to color commutative rings with unity. He considered a simple graph $G(R)$ whose vertices are the elements of $R$, such that two distinct elements $x$ and $y$ are adjacent if and only if $xy=0$. D. F. Anderson and P. S. Livingston \cite{alzdg} modified this definition of the zero-divisor graph by changing the vertex set to the set of nonzero zero-divisors and adjacency remains same.

 In \cite{HM}, R. Hala\v{s} and M. Jukl  introduced the \emph{zero-divisor graph} of a poset with the least element $0$ by considering elements of a poset are the vertices, and there is an edge between two vertices $x$, $y$, if $0$ is the only element lying below $x$ and $y$. They have revealed the similar results of the zero-divisor graph of a poset as Beck did in \cite{Be}.
 
 The definition of the \emph{zero-divisor graph} of a poset, originally introduced in \cite{HM} and later modified by Lu and Wu in \cite{LWP} (see also Joshi \cite{joshi}), is stated below.
 
\begin{definition}
 Let $P$ be a poset with $0$. Define a \emph{zero-divisor} of $P$ to be any element of the set  $Z(P)=\{a\in P$ $|$ there exists $b\in P\setminus\{0\}$ such that $\{a,b\}^\ell=\{0\}\}$. The \emph{zero-divisor graph} of $P$ is the graph $\Gamma(P)$ whose vertices are the elements of $Z(P)\setminus\{0\}$ such that two vertices $a$ and $b$ are adjacent if and only if $\{a,b\}^\ell=\{0\}$, where $\{a,b\}^\ell=\{x\in P\ | \ x \leq a, b\}$.
\end{definition}

In \cite{RHV1}, R.H. Villarreal studied the ideal associated with graphs and named it as an edge ideal. He discovered the class of \emph{Cohen-Macaulay graphs}. As in combinatorial commutative algebra, \emph{Cohen-Macaulay} ring plays a vital role, hence it is important to find a class of simplicial complexes whose Stanley-Reisnar ring (Face Ring) is Cohen-Macaulay. As for a given graph  $\Gamma$, we have the independence complex of $\Gamma.$ A graph $\Gamma$ is said to be \emph{Cohen-Macaulay}, \emph{vertex decomposable}, \emph{Well-covered/Unmixed/Pure} and \emph{Gorenstein} if the independence complex of $\Gamma$ have that property. 
 
In \cite{LW}, A-Ming Liu and Tongsuo Wu proved that Boolean graphs are Cohen–Macaulay.   
In a similar direction, first part of this paper identifies a class of zero-divisor graphs of a Boolean poset possess the Cohen–Macaulay property, and thus generalizes the result of A-Ming Liu and Tongsuo Wu. 

\begin{thm}
If $P$ is a Boolean poset, then $\Gamma(P)$ is Cohen–Macaulay.
\end{thm}

In the later part, we have proved the converse of the above theorem for a special class of posets. This result generalizes Main Theorem of Asir et al. \cite{asir} 
\begin{thm}\label{converse thm}
	Let $ \mathbf{ P}=\prod\limits_{i=1}^{n}P_i$, $(n\geq 3)$, where $P_i$'s are  finite bounded posets such that $Z(P_i)=\{0\}$, $\forall i$ and $2\leq|P_1|\leq|P_2|\leq\dots\leq|P_n|$. Then the following statements are equivalent.
	\begin{enumerate}
		\item $\Gamma(\mathbf{P})$ is Cohen-Macaulay.
		\item $\Gamma(\mathbf{P})$ is well-covered.
		\item  $|P_i|=2$ for all $i=1,2,\ldots,n$.
		\item $\mathbf{P}$ is a Boolean  lattice.
		\item $\mathbf{P}$ is a Boolean  poset.
	\end{enumerate}  
\end{thm}
\section{Preliminaries}

\subsection{Combinatorial Commutative Algebra and Graph Theory:}
A \textit{simplicial complex} $\Delta$ on a set $V$ is a collection of subsets of $V$ such that for every $B \in \Delta$ and every $A \subseteq B$, we also have $A \in \Delta$. The elements of $\Delta$ are called \textit{faces}. The set $V$ is called the \textit{vertex set} of $\Delta$, and its elements are called the \textit{vertices}. 

A face $F$ of $\Delta$ is called a \textit{facet} if there is no face $B \in \Delta$ such that $F \subset B$. The set of all facets of $\Delta$ is denoted by $\mathcal{F}(\Delta)$. The complex $\Delta$ is called a \textit{simplex} if it has exactly one facet.

The \textit{dimension} of a face $A$ is defined as $\operatorname{dim}(A)=|A|-1$. The \textit{dimension} of the simplicial complex $\Delta$ is the maximum dimension of any of its faces and is denoted by $\operatorname{dim}(\Delta)$. A simplicial complex $\Delta$ is called \textit{pure} if all facets of $\Delta$ have the same dimension; otherwise, it is called \textit{nonpure}. A pure simplicial complex is also known as \textit{well-covered}.

\vspace{2mm}

Hereafter, we consider the simplicial complex of independent vertex sets of a graph $\Gamma$, denoted by $\Delta_{\Gamma}$. The faces of $\Delta_{\Gamma}$ are precisely the independent sets of $\Gamma$, and the facets of $\Delta_{\Gamma}$ are the maximal independent sets of $\Gamma$.

Let $S=\mathbb{K}[x_1,\ldots,x_n]$ be the polynomial ring in $n$ variables over a field $\mathbb{K}$, and let $\Delta$ be a simplicial complex on the vertex set $[n]=\{1, 2, \cdots, n\}$. For a subset $F \subseteq [n]$, denote $x_F = \prod_{i \in F} x_i$.

The \textit{Stanley--Reisner ideal} of $\Delta$ over $\mathbb{K}$ is the ideal $I_\Delta \subseteq S$ generated by all squarefree monomials $x_F$ corresponding to subsets $F \notin \Delta$. The corresponding \textit{Stanley--Reisner ring} is defined as
$
\mathbb{K}[\Delta] = S / I_\Delta.
$

A simplicial complex $\Delta$ is said to be \textit{Cohen--Macaulay} (respectively, \textit{Gorenstein}) over $\mathbb{K}$ if its Stanley--Reisner ring $\mathbb{K}[\Delta]$ is a \textit{Cohen--Macaulay ring} (respectively, a \textit{Gorenstein ring}).

Let $\Gamma=(V,E)$ be a finite simple graph with the vertex set $V$. 
The \emph{edge ideal} of $\Gamma$, denoted by $I(\Gamma)$, is the squarefree monomial ideal in the polynomial ring $S=\mathbb{K}[x_1,\ldots,x_n]$ defined as
$
I(\Gamma)=\big( x_i x_j \; : \; \{i,j\} \in E \big).
$
That is, each edge $\{i,j\}$ of the graph corresponds to a quadratic squarefree monomial $x_i x_j$, and the edge ideal is generated by all such monomials. 

The connection with Stanley-Reisner theory arises through the independence complex of $\Gamma$. 
The independence complex, denoted by $\Delta_{\Gamma}$, is the simplicial complex whose faces are the independent sets of $\Gamma$. 
Since the minimal nonfaces of $\Delta_{\Gamma}$ are precisely the edges of $\Gamma$, the Stanley--Reisner ideal of $\Delta_{\Gamma}$ coincides with the edge ideal of $\Gamma$, i.e.,
$
I_{\Delta_{\Gamma}} = I(\Gamma).
$
Consequently, the Stanley--Reisner ring of the independence complex is given by
$
\mathbb{K}[\Delta_{\Gamma}] = S/I(\Gamma).
$

A subset $C$ of the vertex set $V(\Gamma)$ of a graph $\Gamma,$ is called a \textit{vertex cover} of a graph $\Gamma$ if for every $v\in V(\Gamma) $  there exist at least one vertex $u\in C$ such that $v$ is adjacent to $u$ in $\Gamma.$  A vertex cover $C$ of graph $\Gamma$ is called \textit{minimal} if there does not exists any cover of $\Gamma$ which is inside $C.$ A subset $I$ of the vertex set $V(\Gamma) $ is called an \textit{independent set} if the induced subgraph of $\Gamma$ on $I$ is an empty graph (graph without edges). An independent set $I$  is called \textit{maximal} if for any $v\in V(\Gamma)\setminus I$ the induced subgraph on $I\cup \{v\}$ is not an empty graph. Note that the set of independent sets of $\Gamma$ forms a simplicial complex. 

  \begin{definition}
	A graph $\Gamma$ is said to be \emph{well-covered} if all its maximal independent sets have the same cardinality. \emph{Well-covered} graphs are also called as \textit{pure}. 

	Further, $\Gamma$ is called \emph{very well–covered} if it is well–covered, has no isolated vertices, and the number of vertices of $\Gamma$ is exactly twice the size of a maximal independent set. 
\end{definition}

Let $\Gamma$ be a graph and let $v \in V(\Gamma)$. A vertex $w \in V(\Gamma)$ is called a \textit{complement} of $v$ in $\Gamma$, if $v$ is adjacent to $w$, and no vertex is adjacent to both $v$ and $w$, i.e., the edge $v-w$ is not an edge of any triangle in $\Gamma$. In such a case, we write $v \perp w$. Moreover, we say that $\Gamma$ is \textit{complemented} if every vertex has a complement. An \textit{end} is a vertex that is adjacent to precisely one other vertex. For more details, please refer, \cite{VA, LW}.

\subsection{Basics of Poset Theory }
We proceed with the following basics and necessary terminologies given by Devhare, Joshi and
LaGrange in \cite{djl}.  Also, we mention some results that we have used as tools to prove the main results of this paper.

 Let $P$ be a poset. For any subset $A \subseteq P$, the \textit{upper cone} of $A$ is defined as  
 $A^{u} = \{\, b \in P \mid b \ge a \text{ for all } a \in A \,\}$,  
 and the \textit{lower cone} of $A$ is given by  
 $A^{\ell} = \{\, b \in P \mid b \le a \text{ for all } a \in A \,\}$.  
 For a single element $a \in P$, we write $a^{u}$ and $a^{\ell}$ instead of $\{a\}^{u}$ and $\{a\}^{\ell}$, respectively.  
 The notation $A^{u\ell}$ means $(A^{u})^{\ell}$, and similarly one defines $A^{\ell u}$.
 
 An element $a \in P$ is called an \textit{atom} if $a > 0$ and there is no $b \in P$ satisfying $0 < b < a$.  
 The poset $P$ is said to be \textit{atomic} if each nonzero element $b \in P \setminus \{0\}$ contains an atom $a \in P$ with $a \le b$.
 
 A poset $P$ is \textit{bounded} if it has both a least element $0$ and a greatest element $1$.  
 In a bounded poset, an element $a' \in P$ is called a \textit{complement} of $a \in P$ when  
 $\{a,a'\}^{\ell} = \{0\}$ and $\{a,a'\}^{u} = \{1\}$.  
 An element $b \in P$ is a \textit{pseudocomplement} of $a$ if $\{a,b\}^{\ell} = \{0\}$ and $b$ is maximum with this property; that is, $b$ is the pseudocomplement of $a$ exactly when $a^{\perp} = b^{\ell}$, where  
 $a^{\perp} = \{\, x \in P \mid \{x,a\}^{\ell} = \{0\} \,\}$.  
 Pseudocomplements, when they exist, are unique; the pseudocomplement of $a$ is denoted by $a^{*}$.
 
 A bounded poset is called \textit{complemented} (respectively, \textit{pseudocomplemented}) if every element has a complement $a'$ (respectively, a pseudocomplement $a^{*}$).
 
 A poset $P$ is said to be \textit{distributive} if for all $a,b,c \in P$,  
 the equality $\{ \{a\} \cup \{b,c\}^{u} \}^{\ell} = \{ \{a,b\}^{\ell} \cup \{a,c\}^{\ell} \}^{u\ell}$ holds.  
 This notion extends the classical definition of distributivity in lattices: a lattice is distributive in the usual sense if and only if it is distributive as a poset.  
 A bounded poset that is both distributive and complemented is called \textit{Boolean}; see \cite{RH}.  
 Every Boolean algebra is therefore a Boolean poset, though the converse does not necessarily hold.
 
 It is a standard fact that in a Boolean poset, complementation coincides with pseudocomplementation (cf. \cite{VA}, Lemma 2.4).  
 Thus, if $P$ is Boolean, it is pseudocomplemented, and each $x \in P$ has a unique complement $x'$.  
 For any  undefined terminology, we refer the reader to \cite{djl, Gratzer,west}.
 
Next, we present several results on Cohen-Macaulay graphs and the zero-divisor graph of a Boolean poset $P$, which will be essential in proving our main results.

\begin{lem} [{\cite[Lemma 3.1]{MY}}] \label{MY}
	Let $\Gamma$ be a very well-covered graph with $2 h$ vertices. Then the following conditions are equivalent:
	\begin{enumerate}
		\item The graph $\Gamma$ is Cohen-Macaulay.
		\item There is a relabeling of vertices
	 $V(\Gamma)=\left\{x_1, \ldots, x_h, y_1, \ldots, y_h\right\}$ such that the following five conditions hold:
	 \begin{enumerate}
	 	\item $\left\{x_1, \ldots, x_h\right\}$ is a minimal vertex cover of $\Gamma$ and $\left\{y_1, \ldots, y_h\right\}$ is a maximal independent set of $\Gamma$,
	 	\item $\left\{x_1, y_1\right\}, \ldots,\left\{x_h, y_h\right\} \in E(\Gamma)$,
	 	\item if $\left\{z_i, x_j\right\},\left\{y_j, x_k\right\} \in E(\Gamma)$, then $\left\{z_i, x_k\right\} \in E(\Gamma)$ for distinct $i, j, k$ and for $z_i \in\left\{x_i, y_i\right\}$,
	 	\item if $\left\{x_i, y_j\right\} \in E(\Gamma)$, then $\left\{x_i, x_j\right\} \notin E(\Gamma)$,
	 	\item if $\left\{x_i, y_j\right\} \in E(\Gamma)$, then $i \leq j$.
	 \end{enumerate}
\end{enumerate}

\end{lem}
\begin{thm}[{\cite[Theorem 2.4]{VA}}] \label{VA Theorem}
	If $P$ be a Boolean poset, then every vertex of $\GP$ has  the unique  complement in $\GP.$
\end{thm}

\begin{lem}[{\cite[Lemma 2.2]{VA}}] \label{VA Lemma}
Let $P$ be a Boolean poset. Then $(1\neq ) b$ is an atom if and only if
its complement $b'$ is the unique end adjacent to $b$ in $\GP$.
\end{lem}

\section{Cohen-Macaulayness of the zero-divisor graph of a Boolean Poset }
It is well known that every \textit{Cohen--Macaulay graph} is necessarily \textit{unmixed}; see, for example, \cite{RPS, RHV}. Consequently, in order to investigate the Cohen--Macaulayness of the zero-divisor graph $\Gamma(P)$, it is essential first to determine whether $\Gamma(P)$ is \textit{well-covered}. 
The following theorem provides the well-coveredness of $\Gamma(P)$ in the case $P$ is a Boolean poset.

	\textbf{Throughout the paper, unless otherwise specified, $P$ denotes a finite Boolean poset with at least two atoms.}

\begin{thm}\label{thm1}
	The zero-divisor graph of a Boolean poset is  very well-covered.
\end{thm}
\begin{proof}

	Let $\Gamma(P)$ be a zero-divisor graph of a Boolean poset $P$.  
	We prove that all maximal independent vertex sets of $\Gamma(P)$ have the same cardinality, namely $\tfrac{|V(\Gamma(P))|}{2}$.  
	Consequently, $\Gamma(P)$ is very well-covered.

	Since $P$ is a Boolean poset, every element has a unique complement.  
	Hence, $|V(\Gamma(P))| = 2t$, where $t$ is the number of complementary pairs in $P$.  
	Moreover, for each vertex $a \in V(\Gamma(P))$, its complement $a'$ also belongs to $V(\Gamma(P))$.  
	Thus, the vertex set can be decomposed as
	\[
	V(\Gamma(P)) = B \cup B', \quad 
	B = \{a_1,a_2,\ldots,a_t\}, \quad 
	B' = \{a_1',a_2',\ldots,a_t'\},
	\]
	with $B \cap B' = \varnothing$.
	
	It follows that the cardinality of any independent set cannot exceed the number of complementary pairs.  
	In particular, if $S = \{a_1,a_2,\ldots,a_s\}$ is an independent set of $\Gamma(P)$, then
	\[
	|S| \leq t.
	\]
	
	We, now, show that if $|S| < t$, then $S$ can be extended to a larger independent set until $|S| = t$.  
	Since $|S| < t$, there exists at least one complementary pair $\{a_{s+1},a_{s+1}'\}$ such that
	\[
	\{a_{s+1},a_{s+1}'\} \cap S = \varnothing.
	\]
	We claim that at least one of $S \cup \{a_{s+1}\}$ or $S \cup \{a_{s+1}'\}$ remains an independent set.
	
	Suppose on the contrary that neither is independent.  
	Then there exist distinct vertices $a_i,a_j \in S$ such that
	\[
	\{a_i,a_{s+1}\}^\ell = \{0\} \quad \text{and} \quad \{a_j,a_{s+1}'\}^\ell = \{0\},
	\]
	for some $i,j \in \{1,2,\ldots,s\}$.  
	If $a_i = a_j$, then $\Gamma(P)$ would contain a triangle 
	\[
	a_{s+1} - a_i(=a_j) - a_{s+1}' - a_{s+1},
	\]
	contradicting Theorem~\ref{VA Theorem}, since $a_{s+1}'$ is the complement of $a_{s+1}$ in $\Gamma(P)$.
	
	Thus, $a_i \neq a_j$.  
	As $P$ is pseudocomplemented, the relations $\{a_i,a_{s+1}\}^\ell = \{0\}$ and $\{a_j,a_{s+1}'\}^\ell = \{0\}$ imply that $a_i \leq a_{s+1}'$ and $a_j \leq a_{s+1}$.  
	However, since $\{a_{s+1},a_{s+1}'\}^\ell = \{0\}$, we obtain $\{a_i,a_j\}^\ell = \{0\}$, which contradicts the assumption that $a_i$ and $a_j$ are both in the independent set $S$.
	
	Hence, at least one of $a_{s+1}$ or $a_{s+1}'$ can be added to $S$, enlarging it.  
	By repeating this process, $S$ can always be extended until $|S| = t$.  
	
	Therefore, every maximal independent set of $\Gamma(P)$ has size $t = \tfrac{|V(\Gamma(P))|}{2}$, and hence $\Gamma(P)$ is very well-covered.
\end{proof}

\begin{example}

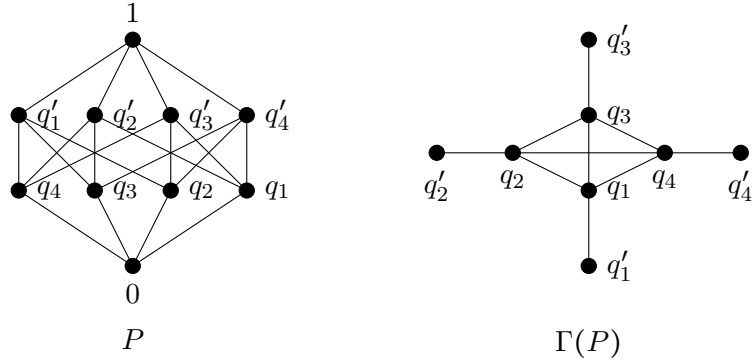
\begin{figure}[h!]
	\centering
	\begin{tikzpicture}[scale=1]
		
		\draw [fill=black] (0,0) circle (.1);
		\draw node [below] at (0,-0.1) { $0$ };
		
		\draw [fill=black] (1.5,1) circle (.1);
		\draw node [right] at (1.6,1) { $q_1$ };
		
		\draw [fill=black] (1.5,2) circle (.1);
		\draw node [right] at (1.6,2) { $q_4'$ };
		
		\draw [fill=black] (0.5,1) circle (.1);
		\draw node [right] at (0.6,1) { $q_2$ };
		
		\draw [fill=black] (-0.5,1) circle (.1);
		\draw node [right] at (-0.4,1) { $q_3$ };
		
		\draw [fill=black] (0,3) circle (.1);
		\draw node [above] at (0,3.1) { $1$ };
		
		\draw [fill=black] (-1.5,1) circle (.1);
		\draw node [right] at (-1.4,1) { $q_4$ };
		
		\draw [fill=black] (-1.5,2) circle (.1);
		\draw node [right] at (-1.4,2) { $q_1'$ };
		
		\draw [fill=black] (-0.5,2) circle (.1);
		\draw node [right] at (-0.4,2) { $q_2'$ };
		
		\draw [fill=black] (0.5,2) circle (.1);
		\draw node [right] at (0.6,2) { $q_3'$ };
		
		\draw (0,0)--(0.5,1)--(0.5,2)--(0,3);
		\draw (0,0)--(-0.5,1)--(-0.5,2)--(0,3);
		\draw (0,0)--(-1.5,1)--(-1.5,2)--(0,3);
		\draw (0,0)--(1.5,1)--(1.5,2)--(0,3);
		
		\draw (-1.5,1)--(-0.5,2);
		\draw (-1.5,1)--(0.5,2);
		\draw (-0.5,1)--(1.5,2);
		\draw (-0.5,1)--(-1.5,2);
		\draw (0.5,1)--(1.5,2);
		\draw (0.5,1)--(-1.5,2);
		\draw (1.5,1)--(0.5,2);
		\draw (1.5,1)--(-0.5,2);
		
		\draw node [below] at (0,-0.7) { $\Huge P$ };
		
		\hspace{1cm}
		
		\draw [fill=black] (5,0) circle (.1);
		\draw node [right] at (5.1,0) { $q_1'$ };
		
		\draw [fill=black] (5,1) circle (.1);
		\draw node [right] at (5.1,1) { $q_1$ };
		
		\draw [fill=black] (5,2) circle (.1);
		\draw node [right] at (5.1,2) { $q_3$ };
		
		\draw [fill=black] (5,3) circle (.1);
		\draw node [right] at (5.1,3) { $q_3'$ };
		
		\draw [fill=black] (4,1.5) circle (.1);
		\draw node [below] at (4,1.4) { $q_2$ };
		
		\draw [fill=black] (6,1.5) circle (.1);
		\draw node [below] at (6,1.4) { $q_4$ };
		
		\draw [fill=black] (7,1.5) circle (.1);
		\draw node [below] at (7,1.4) { $q_4'$ };
		
		\draw [fill=black] (3,1.5) circle (.1);
		\draw node [below] at (3,1.4) { $q_2'$ };
		
		\draw (5,0)--(5,1)--(5,2)--(5,3);
		\draw (3,1.5)--(4,1.5)--(6,1.5)--(7,1.5);
		\draw (4,1.5)--(5,2)--(6,1.5)--(5,1)--(4,1.5);
		
		\draw node [below] at (5,-0.7) { $\Huge \Gamma(P)$ };
		
	\end{tikzpicture}
	\caption{The Boolean poset $P$ and its zero-divisor graph $\Gamma(P)$.}
\end{figure}

In the above figure, $P$ denotes a Boolean poset, and $\Gamma(P)$ is the corresponding zero-divisor graph.  
Consider the simplicial complex of independent sets of $\Gamma(P)$, denoted by $\Delta_{\Gamma(P)}$.  
We have
\[
\Delta_{\Gamma(P)} = \big\langle 
\{q_1,q_2',q_3',q_4'\},
\{q_2,q_1',q_3',q_4'\},
\{q_3,q_1',q_2',q_4'\},
\{q_4,q_1',q_2',q_3'\},
\{q_1',q_2',q_3',q_4'\}
\big\rangle.
\]
It is evident that all facets of $\Delta_{\Gamma(P)}$ have the same cardinality.  
Therefore, $\Delta_{\Gamma(P)}$ is a well-covered simplicial complex. Hence, $\GP$ is well-covered.
	\end{example}

In order to prove that $\Gamma(P)$ is a Cohen-Macaulay graph, we introduce the following definition for  elements of a Boolean poset $P$.

\begin{definition} 
	The \emph{weight} (\(\textit{wt}\)) of an element \(x\) in a finite bounded poset \(P\) is defined as the number of atoms lying below \(x\).  
\end{definition}  

The weight of an element \(x\) will be denoted by \(\operatorname{wt}(x)\). Clearly, every atom has weight \(1\), and  the \textit{weight of a poset} \(P\), denoted by \(\operatorname{\textit{wt}}(P)\), is defined as the weight of the greatest element \(1 \in P\).  Furthermore, the \textit{weight of a vertex} \(v \in \Gamma(P)\) is defined as the weight of the element \(v\) in \(P\). 

Let $B_i$ be the set of all elements in $P$ with weight $k-i$, where $k$ is the number of all atoms of a Boolean poset $P$. One can verify that $B_1$ and $B_3$ are nonempty and $B_2$ is empty in the case of a Boolean poset depicted in Figure 1.

	\begin{lem}\label{weight lem}
		Let $P$ be a Boolean poset and let $x \in P$.  
		Then $x^{\prime}$ is the complement of $x$ if and only if 
		$
		\operatorname{wt}(x^{\prime}) = \operatorname{wt}(P)-\operatorname{wt}(x)
		\quad\text{and}\quad
		\{x,x^{\prime}\}^{\ell}=\{0\}.
		$
	\end{lem}
	
	\begin{proof}
		Let $P$ be a Boolean poset with atoms $\{p_1, p_2, \dots, p_k\}$, and let
		$x \in P$ with $\operatorname{wt}(x) = t$.  
		Since every element in $P$ has a unique complement, say $x'$.
		As $\operatorname{wt}(x) = t$, consider the set
		$\{p_{i_1}, p_{i_2}, \dots, p_{i_t}\}$ of $t$ atoms lying below $x$, where
		$\{i_1, i_2, \dots, i_t\} \subseteq \{1, 2, \dots, k\}$.  
		Similarly, let
		$\{p_{j_1}, p_{j_2}, \dots, p_{j_{\operatorname{wt}(x')}}\}$
		be the set of atoms lying below $x'$.
		
		Now, suppose that $\operatorname{wt}(x') > k - t$.  
		Then there exists at least one atom, say $p_s$, that is common to both sets
		$\{p_{i_1}, p_{i_2}, \dots, p_{i_t}\}$ and
		$\{p_{j_1}, p_{j_2}, \dots, p_{j_{\operatorname{wt}(x')}}\}$.  
		Hence $p_s \leq x$ and $p_s \leq x'$, which implies that
		$\{p_s\} \subseteq \{x, x'\}^{\ell} \neq \{0\}$,
		contradicting the fact that $x'$ is the complement of $x$.  
		Therefore, $\operatorname{wt}(x') \leq k - t$.
		
		Next, suppose $\operatorname{wt}(x') < k - t$.  
		Then there exists an atom $p_q$ that is not below either $x$ or $x'$.  
		Consequently, $\{x, p_q\}^{\ell} = \{0\} = \{x', p_q\}^{\ell}$, implying that the edge $x$–$x'$ forms an edge of a triangle, contradicting Theorem~\ref{VA Theorem}.  
		Thus, $\operatorname{wt}(x') = k - t = \operatorname{wt}(P) - \operatorname{wt}(x)$.  Clearly, $\{x, x'\}^\ell=\{0\}$. Thus the condition is verified.
		
		Conversely, suppose
		$\operatorname{wt}(x') = \operatorname{wt}(P)-\operatorname{wt}(x) = k - t$
		and $\{x,x'\}^{\ell} = \{0\}$.  
		Then no atom lies below both $x$ and $x'$, and the total number of atoms below
		$x$ and $x'$ is $t+(k-t)=k$.
		
		Finally, let $y \in \{x, x'\}^{u}$.  
		Then $y$ must lie above all $k$ atoms of $P$, and hence $y = 1$.  
		Therefore, $\{x, x'\}^{u} = \{1\}$, showing that $x'$ is a complement of $x$.
		Since $P$ is Boolean, this complement is unique.
	\end{proof}

Let $P$ be a poset with $0$. The poset $P$ is called \textit{weakly section semi-complemented} if, for $a, b \in P$ and $a<b$, there exists $c \in P$ such that $0<c \leq b$ and $\{a,c\}^{\ell}=\{0\}$. Further, $P$ is said be \textit{section semi-complemented } if $a, b \in P$ and $b \not \leq a$, there exists $c \in P$ such that $0<c \leq b$ and $\{a,c\}^{\ell}=\{0\}$. Clearly, every section semi-complemented poset is weak section semi-complemented. A finite bounded section semi-complemented poset is called \textit{atomistic.}

\begin{lem}\label{wssc}
	Every uniquely complemented poset is weakly section semi-complemented.
\end{lem}
\begin{proof}
	Let $P$ be a uniquely complemented poset and let $a < b$ for $a, b \in P$.
	Consider the complement $a'$ of $a$ in $P$. If $\{a', b\}^\ell=\{0\}$, then $a$ will have two complements namely $a'$ and $b$, as we have $\{1\}=\{a, a'\}^u \subseteq \{b, a'\}^u$. Hence $\{a', b\}^\ell\not=\{0\}$. Choose any $c \in \{a', b\}^\ell=\{0\}$. Then we have $0< c \leq b$ and $\{c,a\}^\ell=\{0\}$. This shows that every uniquely complemented poset is weakly section semi-complemented. 
\end{proof}
However, the converse is not true; see   Waphare and Joshi \cite[Figure 1]{WJ}. 

As an immediate consequence of the above lemma, we have 

\begin{cor}\label{corwssc}
	Let $P$ be a Boolean poset and $a, b \in P$ such that $a < b$. Then $\operatorname{wt}(a) < \operatorname{wt}(b)$.
\end{cor}
\begin{proof}
	In the proof of Lemma \ref{wssc},  take  any atom below $c$, say $p$.  Then $p$ will be below $b$ but not $a$. Thus $\operatorname{wt}(a) < \operatorname{wt}(b)$.
\end{proof}

Figure 1 given in Waphare and Joshi \cite{WJ} shows that a uniquely complemented poset need not be section semi-complemented. However, if we consider a Boolean poset then the assertion is true.

\begin{lem}
	Every Boolean poset $P$ is section semi-complemented. In particular, if $P$ is finite, then it is atomistic.
\end{lem}
\begin{proof}
	Suppose $a \nleq b$. We claim that $\{a,b'\}^{\ell} \setminus \{0\} \neq \varnothing$. 
	Suppose on the contrary that $\{a,b'\}^{\ell} = \{0\}$. Then we have $a \leq b''=b$, a contradiction, as $P$ is a Boolean poset and hence it is pseudocomplemented. 
	Hence
	$\{a,b'\}^{\ell} \setminus \{0\} \neq \varnothing$. 
	Hence there exists a nonzero element $c \in P$ such that $c \leq a$ and $c \leq b'$; hence $c \nleq b$.
	
	Since $P$ is finite, there exists an \emph{atom} $p$ with $0 < p \leq c$. 
	Then $p \leq a$ and $p \leq b'$, that is, $\{p, b\}^\ell=\{0\}$. 	
\end{proof}

With the above preparation, we are ready to prove the following main theorem.

\begin{thm}\label{CM}
	If $P$ is a Boolean poset, then $\Gamma(P)$ is Cohen--Macaulay.
\end{thm}
\begin{proof}
	Let $P$ be a Boolean poset and $\operatorname{wt}(P)=k.$
	By Theorem \ref{thm1}, every maximal independent set have the cardinality equals
	to the number of complementary pairs, say, $t$ in $P.$
	Also, note that $t=\frac{|V(\Gamma(P))|}{2}$.
	If $k=2$, then $\Gamma(P)$ is $K_2$ and in this case the result is true.
	Hence assume that $k>2$.
	
	First, let us construct a maximal independent set $B$ of $\Gamma(P).$
	
	Suppose that the weight of the poset $k$ is odd.
	Then for $1\leq i\leq \frac{k-1}{2}$, define $B_i$ be the set of all elements in
	$P$ with weight $k-i.$
	
	Let $B_1=\{y_1,y_2,\dots,y_k\}$ be the set of elements of weight $k-1$.
	Clearly, $B_1\neq \emptyset$, as $P$ has at least two atoms and contains all the
	dual atoms of $P$.
	
	As there is no guaranty that $P$ contains an elements of weight $k-i$ for
	$2\leq i\leq k-2$, $B_i$ may or may not be empty for $2\leq i\leq k-2$.
	
	Now, suppose that $k$ is an even.
	For $i=\frac{k}{2}$, define the set $B_i=\widehat{B}$.
	
	If there are complementary pairs of weight $\frac{k}{2}$, then the set
	$B_i=\widehat{B}$ contains any one element from the each complementary pairs.
	
	Now,
	\[
	B =
	\begin{cases}
		B_1 \cup B_2 \cup \dots \cup B_{\frac{k-1}{2}},
		& \text{if $k$ is odd},\\[4pt]
		B_1 \cup B_2 \cup \dots \cup B_{\frac{k-2}{2}} \cup \widehat{B},
		& \text{if $k$ is even.}
	\end{cases}
	\]
	
	Now, we show that the set $B$ is a maximal independent set.
	
	Suppose that $k$ is even.
	Observe  that `the weight of each element in  $B$ is greater than or equal to 
	$\frac{k}{2}$.
	Let $x,y\in B$.
	Then $\operatorname{wt}(x)+\operatorname{wt}(y)\geq k$.

		If equality holds, then necessarily
		$\operatorname{wt}(x)=\operatorname{wt}(y)=\frac{k}{2}$.
		In order for $x$ and $y$ to be adjacent, we must have
		$\{x,y\}^{\ell}=\{0\}$, in which case Lemma~\ref{weight lem} implies that
		$x$ and $y$ form a complementary pair.

	However, by the construction, $B$ contains one element from each complementary
	pair having weight $\frac{k}{2}$, a contradiction.
	Thus equality does not hold.

		If $\operatorname{wt}(x)+\operatorname{wt}(y)> k$, then the sets of atoms lying
		below $x$ and $y$ intersect, and hence $\{x,y\}^{\ell}\neq \{0\}$.
		Hence $x$ and $y$ are not adjacent.
	
	Now, assume that $k$ is odd.
	Then all the elements of $B$ are of weight greater than or equal to
	$\frac{k+1}{2}$.
	Thus for any two elements $x,y\in B$ satisfy
	$\operatorname{wt}(x)+\operatorname{wt}(y)> k$. 	
		Therefore, the sets of atoms below $x$ and $y$ intersect, which implies
		$\{x,y\}^{\ell}\neq\{0\}$.
	 	Thus, $x$ and $y$ are not adjacent.
	
	Thus, $B$ is an independent set.
	\vskip 4pt
	
	Now, we prove $B$ is a maximal independent set.
	For that, let $y\in V(\Gamma(P))\setminus B$.
	Clearly, $0<\operatorname{wt}(y)\leq \frac{k}{2}$, if $k$ is even.
	
	If $\operatorname{wt}(y)=\frac{k}{2}$, then by Lemma \ref{weight lem} and the
	construction of $B$, $\operatorname{wt}(y^{\prime})=\frac{k}{2}$ and
	$y^{\prime}\in B$.

		Since $y'$ is the complement of $y$, we have $\{y,y'\}^{\ell}=\{0\}$.
		So,  $y$ and $y^{\prime}$ are adjacent.
	
	If $\operatorname{wt}(y)=l< \frac{k}{2}$, clearly, $y^{\prime}$ is in $P$ with
	$\operatorname{wt}(y^{\prime})=k-l$.
	As $l<\frac{k}{2}$, $\operatorname{wt}(y^{\prime})>\frac{k}{2}$, by the
	construction, $y^{\prime}\in B$. 
	
		Again, $\{y,y'\}^{\ell}=\{0\}$, so $y$ is adjacent to $y'$.

	So, we can not extend the independent set $B$, when $k$ is even.
	
	Now, assume that $k$ is odd.
	Then $B$ contains all elements of weight greater than or equal to
	$\frac{k+1}{2}$.
	Thus the element $y\in V(\Gamma(P))\setminus B$ has weight at most
	$\frac{k-1}{2}$. 	
		By Lemma \ref{weight lem}, $y$ has a unique complement $y'\in B$ and
		$\{y,y'\}^{\ell}=\{0\}$.
		Thus $y$ is adjacent to $y'$.
	
	Thus $B$ is the maximal independent set of $\Gamma(P)$.
	\vskip 4pt
	As $B$ admits no natural order, we impose one by listing the elements of
	$B_1, B_2, \ldots$ successively, and, when $k$ is even, listing $\widehat{B}$ last.
	The ordering within each subset is arbitrary and serves only for labelling purposes.
		
	Now, let's give relabelling to the elements of $B$.
	Label
	\[
	B=\{x_1^{\prime},x_2^{\prime},\dots,x_t^{\prime}\}.
	\]
	
	Now, let the set
	\[
	A=V(\Gamma(P))\setminus B=\{x_1,x_2,\dots,x_t\}
	\]
	such that $x_i^{\prime}$ is the complement of $x_i$ for all
	$i=1,2,\dots,t$ and $V(\Gamma(P))=A\cup B$.
	
	Finally, to prove that $\Gamma(P)$ is Cohen--Macaulay, the above labeling must
	satisfy the conditions given in Lemma \ref{MY}.
	
	\textit{Condition (a):}
	Since the complement of any maximal independent set is a minimal vertex cover,
	the set $A$ is a minimal vertex cover of $\Gamma(P).$
	
	\textit{Condition (b):}
	Since $\{x_i,x_i^{\prime}\}^{\ell}=\{0\}$, we have $x_i$ is adjacent to
	$x_i^{\prime}$ in $\Gamma(P)$ for each $i=1,2,\dots,t$.
	
	\textit{Condition (c):}  Assume that $\{z_i,x_j\}, \{x_j^{\prime},x_k\} \in E(\Gamma(P))$ for $z_i \in \{x_i, x_i^\prime\}$. Then $\{z_i,x_j\}^\ell =\{0\}=\{x_k, x_j^\prime\}^\ell$.  Since $P$ is Boolean, it is pseudocomplemented and unique complementation coincides with the pseudocomplementation, we have $x_j^{\prime}$ is pseudocomplement of $x_j$. Further,  $\{z_i,x_j\}^\ell =\{0\}$, we get  $z_i\leq x_j^{\prime}$. Similarly,  and $x_k\leq x_j^{\prime\prime}=x_j.$ This implies that  $\{z_i,x_k\}^{\ell}=\{0\}.$ Therefore, $\{z_i,x_k\}$ is an edge in $\Gamma(P)$.
	
	\textit{Condition (d):}
	Suppose $\{x_i,x_j^{\prime}\}\in E(\Gamma(P))$.
	Then $\{x_i,x_j^{\prime}\}^\ell=\{0\}$.
	This gives $x_i \leq x_j^{\prime\prime}=x_j$.
	Hence $\{x_i,x_j\}\notin E(\Gamma(P))$.
	
	\textit{Condition (e):}
	Assume that $\{x_i,x_j^{\prime}\}\in E(\Gamma(P))$. Then $\{x_i,x_j^{\prime}\}^\ell=\{0\}$. 
	
	We claim that $i\leq j$.
	
	Now, if possible let $i>j$.
	By the relabeling $\operatorname{wt}(x_i^\prime)
	\leq \operatorname{wt}(x_j^\prime)$.
	As $x_i^\prime$ is the unique complement of $x_i$ in $P$,
	$\operatorname{wt}(x_i)+\operatorname{wt}(x_i^\prime)=k$.
	By the above inequality,
	$\operatorname{wt}(x_i)+\operatorname{wt}(x_j^\prime)\geq k$.

		If equality holds, then together with $\{x_i,x_j'\}^{\ell}=\{0\}$,
		Lemma \ref{weight lem} implies that $x_j^\prime$ is the complement of $x_i$,
		contradicting uniqueness.

		If $\operatorname{wt}(x_i)+\operatorname{wt}(x_j^\prime)> k$, then the sets of
		atoms below $x_i$ and $x_j'$ intersect, and hence
		$\{x_i,x_j^\prime\}^{\ell}\neq \{0\}$,
		contradicting $\{x_i,x_j^\prime\}\in E(\Gamma(P))$.
		Hence, $i\leq j$.
	
	Thus all the conditions of Lemma \ref{MY} are satisfied, and $\Gamma(P)$ is
	Cohen--Macaulay.
	\end{proof}

  As an immediate consequence of the above result, we have the following results proved by A-Ming Liu and Tongsuo Wu in \cite{LW}.

\begin{cor}
	The zero-divisor graph of a Boolean algebra (equivalently Boolean ring) is Cohen–Macaulay.
\end{cor}

Since the simplicial complex of independent sets of a graph $G$ coincides with the simplicial complex of cliques in its complement $G^c$, we further deduce the following.

\begin{cor}
	The complement $\Gamma^c(P)$ of the zero-divisor graph of a Boolean poset $P$ is Cohen–Macaulay with its clique complex.
\end{cor}
In general, the converses of Theorems~\ref{thm1}, and \ref{CM} do not hold.  
The following example serves as a counterexample for the same.

\begin{example}
	Consider a bounded  poset $P$, having three atoms say $a,b,c$ covered by $1,$  which is not a Boolean poset, but the corresponding zero-divisor graph is the complete graph $K_3$. Hence, it is well-covered and Cohen-Macaulay. 
	
\end{example}

Now, we prove the converse of Theorem \ref{CM} for a special class of posets.
For this purpose, we introduce the notations and definitions given in~\cite{KJ}, which will be necessary for proving the forthcoming results.

 The  direct product of posets $P_1,\dots,P_n$ is the poset
$\textbf{P}=\prod\limits_{i=1}^nP_i$ with $\leq$ defined such that $a\leq b$ in $\textbf{P}$ if and only if
$a_{i}\leq b_{i}$ (in $P_i$) for every $i\in\{1,\dots,n\}$. For any
$\emptyset\neq A\subseteq \prod\limits_{i=1}^nP_i$, note that
$A^u=\{b\in\prod\limits_{i=1}^nP_i$ $|$ $b_{i}\geq a_{i}$ for every $a\in A$
and $i\in\{1,\dots,n\}\}$. Similarly, $A^\ell=\{b\in\prod\limits_{i=1}^nP_i$
$|$ $b_{i}\leq a_{i}$ for every $a\in A$ and $i\in\{1,\dots,n\}\}$.

{\par  Throughout this section, $ \mathbf{ P}=\prod\limits_{i=1}^{n}P_i$, $(n\geq 2)$, where $P_i$'s are  finite bounded posets such that $Z(P_i)=\{0\}$, $\forall i$ and $2\leq|P_1|\leq|P_2|\leq\dots\leq|P_n|$. Since $P_{i}$'s are finite bounded poset, $P_{i}$'s are atomic. Note that $Z(P_i)=0$ if and only if $P_i$ has the unique atom. Further, assume that $q_{_i}\in P_i$ is the unique atom of $P_i$ for every $i$. 
	
\textit{ All the elements of $ \mathbf{ P}$ are denoted by bold letters}. 

Let $\mathbf{x}=(x_{1},x_{2},\dots,x_{n})\in \mathbf{ P} $, where $x_i\in P_i$. In particular, $\mathbf{q}_{_1},\mathbf{q}_{_2},\dots, \mathbf{q}_{_n}$ are the only atoms of $ \mathbf{ P}$. That is, $\mathbf{q}_{_{i}}=(0,\dots,0,q_{_i},0,\dots,0)$. We observe that $(0,0,\dots,0),(1,1,\dots,1)\in \mathbf{ P}$ are the least and the greatest elements of $ \mathbf{ P}$ denoted by $\mathbf{0}$ and $\mathbf{1}$ respectively. We set $D = \mathbf{P}\setminus Z(\mathbf{P}).$ The elements $d \in D $ are the dense
elements of $\mathbf{P}$.

\begin{thm}
	 Let $ \mathbf{ P}=\prod\limits_{i=1}^{n}P_i$, $(n=2)$, where $P_i$'s are  finite bounded posets such that $Z(P_i)=\{0\}$, $\forall i$ and $2\leq|P_1|\leq|P_2|$. Then the following statements are equivalent.
	 \begin{enumerate}
	 	\item $\Gamma(\mathbf{P})$ is Cohen-Macaulay.
	 	\item $|P_1|=|P_2|=2$, therefore,  $\mathbf{P}$ is a Boolean lattice
	 \end{enumerate}
\end{thm}
\begin{proof}
	$(1)\implies (2)$: 	Observe that, $\Gamma(\mathbf{P})$ is a complete bipartite graph $K_{|P_1|-1,|P_2|-1}$ which is \emph{Cohen-Macaulay} only if  $|P_1|=|P_2|=2.$
	
	$(2)\implies (1)$:
	If $|P_1|=|P_2|=2$, then $\Gamma(\mathbf{P})=K_2.$ Hence, it is \emph{Cohen-Macaulay.}
\end{proof}

\begin{lem}\label{MIS}  Let $ \mathbf{ P}=\prod\limits_{i=1}^{n}P_i$, $(n\geq 3)$, where $P_i$'s are  finite bounded posets such that $Z(P_i)=\{0\}$, $\forall i$ and $2\leq|P_1|\leq|P_2|\leq\dots\leq|P_n|$. Let $\mathbf{q}_1, \mathbf{q}_2, \ldots, \mathbf{q}_n$ be the all atoms of $\mathbf{P}$. For each $1\leq i<j<k\leq n$, define \[ \mathfrak{J}_{i,j,k} \;=\; \Bigl\{\, \mathbf{x} \in \mathbf{P} \;\Bigm|\; \mathbf{x} \in \{\mathbf{q}_i,\mathbf{q}_j\}^u \;\;\text{or}\;\; \mathbf{x} \in \{\mathbf{q}_j,\mathbf{q}_k\}^u \;\;\text{or}\;\; \mathbf{x} \in \{\mathbf{q}_i,\mathbf{q}_k\}^u \Bigr\}\setminus D . \] Then $\mathfrak{J}_{i,j,k}$ is a maximal independent set in $\Gamma(\mathbf{P})$. \end{lem}
\begin{proof}
	Let $\mathbf{x},\mathbf{y}\in \mathfrak{J}_{i,j,k}$. We prove the result by using following cases.  
	
	\noindent\textbf{Case I.} Without loss of generality, assume that  $\mathbf{x},\mathbf{y}\in \{\mathbf{q}_i,\mathbf{q}_j\}^u$.  Then \mbox{$\{\mathbf{q}_i,\mathbf{q}_j\}\subseteq \{\mathbf{x},\mathbf{y}\}^{\ell}\neq \{\mathbf{0}\}.$  }
	
	\noindent\textbf{Case II.} Without loss of generality, assume that  $\mathbf{x}\in \{\mathbf{q}_i,\mathbf{q}_j\}^u$ and $\mathbf{y}\in \{\mathbf{q}_j,\mathbf{q}_k\}^u$, then $\{\mathbf{q}_j\}\subseteq \{\mathbf{x},\mathbf{y}\}^{\ell}\neq \{\mathbf{0}\}.$  
	
	Thus, in every case, $\{\mathbf{x},\mathbf{y}\}^{\ell}\neq \{\mathbf{0}\}$, which means no two vertices of $J_{i,j,k}$ are adjacent in $\Gamma(\mathbf{P})$. Hence $\mathfrak{J}_{i,j,l}$ is an independent set.  
	\vskip 4pt 
	Now, to prove maximality of $\mathfrak{J}_{i,j,k}$, let $\mathbf{w}\in V(\Gamma(\mathbf{P}))\setminus \mathfrak{J}_{i,j,k}$.  
	Then $\mathbf{w}\notin \{\mathbf{q}_i,\mathbf{q}_j\}^u$,\linebreak $\mathbf{w}\notin \{\mathbf{q}_j,\mathbf{q}_k\}^u$, and $\mathbf{w}\notin \{\mathbf{q}_i,\mathbf{q}_k\}^u$.  	Without loss of generality, assume that $ \mathbf{q}_i \nleq \mathbf{w} $ and $\mathbf{q}_j \nleq \mathbf{w}.$ So, $x_i=0=x_j$ for $\mathbf{w}=(x_1,x_2,\dots,x_n).$
	Choose an element $\mathbf{x}=(0,\dots,0,q_i,\dots,q_j,0,\dots,0)$. Then $\mathbf{x}\in \{\mathbf{q}_i,\mathbf{q}_j\}^u$. Hence  $\mathbf{x} \in \mathfrak{J}_{i,j,k}$.  Further, $\{\mathbf{x},\mathbf{w}\}^\ell = \{0\}$.  
	This means $\mathbf{x}$ and $\mathbf{w}$ are adjacent in $\Gamma(\mathbf{P})$.  Therefore, $\mathbf{w}$ is adjacent to at least one vertex of $\mathfrak{J}_{i,j,k}$, and hence, $\mathfrak{J}_{i,j,k}\cup\{\mathbf{w}\}$ is not an independent set.  	Hence, $\mathfrak{J}_{i,j,k}$ is a maximal independent set in $\Gamma(\mathbf{P})$.  
\end{proof}

The proof of the next result follows the same ideas as the proof of Theorem~3.5 presented by Asir et al.~\cite{asir}, but rewritten in the language of poset theory. For the sake of completeness, we include the full argument here.

\begin{thm}\label{well}
	Let $ \mathbf{ P}=\prod\limits_{i=1}^{n}P_i$, $(n\geq 3)$, where $P_i$'s are  finite bounded posets such that $Z(P_i)=\{0\}$, $\forall i$ and $2\leq|P_1|\leq|P_2|\leq\dots\leq|P_n|$. Then the zero-divisor graph $\Gamma(\mathbf{P})$ is well-covered if and only if $|P_i|=2$ for all $i=1,2,\ldots,n$.
\end{thm}
\begin{proof}
	If $|P_i|=2$ for all $i=1,2,\ldots,n$, then $P$ is a Boolean poset. By Theorem~\ref{thm1}, it follows that $\Gamma(P)$ is  well-covered.  
	
	Conversely, suppose that $\Gamma(\mathbf{P})$ is   well-covered,  i.e., all its maximal independent sets have the same cardinality.  
	
	Let $\mathbf{q}_{_1},\mathbf{q}_{_2},\dots, \mathbf{q}_{_n}$ be the atoms of $\mathbf{P}$. For each atom $\mathbf{q}_i$, consider the maximal independent set
	\[
	\mathfrak{J}_i=\{\mathbf{q}_i\}^u\setminus D=\{(x_1,\ldots,x_k)\in \mathbf{P}\mid x_i\neq 0\}\setminus D,
	\]
	where $D$ denotes the set of dense elements of $\mathbf{P}$. Since 
	$
	|D|=(|P_1|-1)(|P_2|-1)\cdots(|P_n|-1),
	$
	we obtain
	$
	|\mathfrak{J}_i|=\bigl(|P_1||P_2|\cdots(|P_i|-1)\cdots |P_n|\bigr)-\bigl((|P_1|-1)(|P_2|-1)\cdots(|P_n|-1)\bigr).
	$
	As $\Gamma(\mathbf{P})$ is well-covered, $|\mathfrak{J}_i|=|\mathfrak{J}_j|$ for all $i\neq j$. This equality simplifies to
	\[
	(|P_i|-1)|P_j|=(|P_j|-1)|P_i| \quad \implies \quad |P_i|=|P_j|.
	\]
	Hence $|P_1|=|P_2|=\dots=|P_n|=\alpha$ (say).  
	
	It remains to show that $\alpha=2$. By  Lemma \ref{MIS}, consider the maximal independent set
	\[ \mathfrak{J}_{1,2,3} \;=\; \Bigl\{\, \mathbf{x} \in \mathbf{P} \;\Bigm|\; \mathbf{x} \in \{\mathbf{q}_1,\mathbf{q}_2\}^u \;\;\text{or}\;\; \mathbf{x} \in \{\mathbf{q}_2,\mathbf{q}_3\}^u \;\;\text{or}\;\; \mathbf{x} \in \{\mathbf{q}_1,\mathbf{q}_3\}^u \Bigr\}\setminus D . \]
	Writing the set $\mathfrak{J}_{1,2,3}=A\cup B\cup C$, where
	$
	A=\{\mathbf{x}\in \mathbf{P}\mid \mathbf{x}\in \{\mathbf{q}_1,\mathbf{q}_2\}^u\}\setminus D,\quad \linebreak
	B=\{\mathbf{x}\in \mathbf{P}\mid \mathbf{x}\in \{\mathbf{q}_2,\mathbf{q}_3\}^u\}\setminus D,\quad \text{and }
	C=\{\mathbf{x}\in \mathbf{P}\mid \mathbf{x}\in \{\mathbf{q}_1,\mathbf{q}_3\}^u\}\setminus D.
	$
	By the principle of Inclusion-Exclusion,
	\[
	|\mathfrak{J}_{1,2,3}|=|A|+|B|+|C|-|A\cap B|-|A\cap C|-|B\cap C|+|A\cap B\cap C|.
	\]
	Note that
	\[
	|A|=|B|=|C|=(\alpha-1)^2\alpha^{n-2}-(\alpha-1)^n,
	\]
	and
	\[
	|A\cap B|=|A\cap C|=|B\cap C|=|A\cap B\cap C|=(\alpha-1)^3\alpha^{n-3}-(\alpha-1)^n.
	\]
	Thus
	\[
	|\mathfrak{J}_{1,2,3}|=3\big((\alpha-1)^2\alpha^{n-2}-(\alpha-1)^n\big)-2\big((\alpha-1)^3\alpha^{n-3}-(\alpha-1)^n\big).
	\]
	
	On the other hand, we know that
	\[
	|\mathfrak{J}_1|=(\alpha-1)\alpha^{n-1}-(\alpha-1)^n.
	\]
	Since, $\Gamma(\mathbf{P})$ is unmixed, $|\mathfrak{J}_{1,2,3}|=|\mathfrak{J}_1|$. Equating the two expressions, we have
	\[
	3(\alpha-1)^2\alpha^{n-2}-2(\alpha-1)^3\alpha^{n-3}=(\alpha-1)\alpha^{n-1}.
	\]
	Dividing through by $\alpha^{n-3}>0$, we obtain
	\[
	3(\alpha-1)\alpha-2(\alpha-1)^2-\alpha^2=0,
	\]
	which simplifies to
	$ 	\alpha-2=0. $
	Thus, $\alpha=2$, and therefore, $|P_1|=|P_2|=\dots=|P_n|=\alpha=2$, completing the proof.
\end{proof}

With this, we give a proof of 
Theorem, \ref{converse thm} that characterizes the Cohen-Macaulayness of the zero-divisor graph $\Gamma(\mathbf{P})$.\\

\noindent\textbf{Proof of Theorem \ref{converse thm} :} 
\begin{proof}
	$(1)\implies (2)$:
	It is well-known that, every Cohen-Macaulay graph is well-covered. 
	
	$(2)\implies (3)$:
	follows from Theorem \ref{well}.
	
	$(3)\implies (4)$:
	Let  $ \mathbf{ P}=\prod\limits_{i=1}^{n}P_i$, $(n\geq 3)$, where $|P_i|=2$ for all $i=1,2,\ldots,n.$ Hence each $P_i$ is a $2$-element  chain. Thus, $\mathbf{P}$  is a Boolean lattice.
	
	$(4)\implies (5)$: 
	Every Boolean lattice is Boolean poset.
	
	$(5)\implies (1)$:
	 From Theorem \ref{CM}, $\Gamma(\mathbf{P})$ is Cohen-Macaulay.
\end{proof}

\noindent\textbf{Funding:}
First author: None.\\
Second author: Supported by Department of Science and Technology, Science and Engineering Research Board, Government of India under the scheme CRG/2022/002184.

\noindent\textbf{Conflict of interest:} The authors declare that there is no conflict of
interests regarding the publishing of this paper.

\noindent\textbf{Authorship Contributions :} Both authors contributed to the study on the Cohen-Macaulayness of the zero-divisor graph of a Boolean poset. Both authors read and approved the final version of the manuscript.

\noindent \textbf{Data Availability Statement :} Data sharing does not apply to this article, as no datasets were generated or analyzed during the current study.

\noindent\textbf{Acknowledgment:} The authors sincerely thank the referee for a meticulous evaluation of the manuscript. The constructive feedback provided was instrumental in improving the final presentation of the paper.

\end{document}